\documentclass[10pt,reqno]{amsart}
\usepackage{hyperref}
\usepackage{float}
\usepackage{mathtools}
\usepackage[all]{xy}

\usepackage{tikz-cd}
\usepackage{quiver}
\usepackage{amscd,amssymb,amsmath,latexsym,bm}
\usepackage[mathcal,mathscr]{euscript}
\usepackage{lipsum}
\usepackage{amsfonts}
\usepackage{graphicx}
\usepackage{epstopdf}
\usepackage{algorithmic}
\usepackage{hyperref}
\usepackage{xcolor}
\usepackage{upgreek}
\usepackage{enumerate}
\usepackage{ulem}
\usepackage{stmaryrd}
\usepackage{gensymb}			
\usepackage[utf8]{inputenc}
 
\DeclareMathOperator{\Iso}{\mathrm{Iso}}

\newtheorem{theorem}{Theorem}[section]

\newtheorem{corollary}[theorem]{Corollary}
\newtheorem{proposition}[theorem]{Proposition}
\newtheorem{remark}[theorem]{Remark}
\newtheorem{definition}[theorem]{Definition}
\newtheorem{example}[theorem]{Example}

\numberwithin{equation}{section}
\numberwithin{figure}{section}

\newcommand{\CM}{{\mathbb C}}
\newcommand{\NM}{{\mathbb N}}
\newcommand{\PM}{{\mathbb P}}

\newcommand{\RM}{{\mathbb R}}
\newcommand{\SM}{{\mathbb S}}

\newcommand{\ZM}{{\mathbb Z}}

\newcommand{\KM}{{\mathbb K}}

\newcommand{\EM}{{\mathbb E}}

\newcommand{\Bb}{{\mathcal B}}

\newcommand{\Ff}{{\mathcal F}}
\newcommand{\Gg}{{\mathcal G}}

\newcommand{\Cc}{{\mathcal C}}

\newcommand{\Ll}{{\mathcal L}}

\newcommand{\Hh}{{\mathcal H}}

\begin{document}
\title[Classifying the Dynamics of Architected Materials]{Classifying the Dynamics of Architected Materials by Groupoid Methods}

\author{Bram Mesland}

\address{Mathematisch Instituut \\ Universiteit Leiden  \\ Niels Bohrweg 1 \\ 2333CA Leiden, Netherlands
\\
\href{mailto:b.mesland@math.leidenuniv.nl}{b.mesland@math.leidenuniv.nl}}

\author{Emil Prodan}

\address{Department of Physics and
\\ Department of Mathematical Sciences 
\\Yeshiva University 
\\New York, NY 10016, USA \\
\href{mailto:prodan@yu.edu}{prodan@yu.edu}}

\date{\today}

\begin{abstract} 
We consider synthetic materials consisting of self-coupled identical resonators carrying classical internal degrees of freedom. The architecture of such material is specified by the positions and orientations of the resonators. Our goal is to calculate the smallest $C^\ast$-algebra that covers the dynamical matrices associated to a fixed architecture and adjustable internal structures. We give the answer in terms of a groupoid $C^\ast$-algebra that can be canonically associated to a uniformly discrete subset of the group of isometries of the Euclidean space. Our result implies that the isomorphism classes of these $C^\ast$-algebras split these architected materials into classes containing materials that are identical from the dynamical point of view. 
\end{abstract}

\thanks{This work was supported by the U.S. National Science Foundation through the grants DMR-1823800 and CMMI-2131760 and by U.S. Army Research Office through contract W911NF-23-1-0127.}

\maketitle


\setcounter{tocdepth}{1}

\section{Introduction and main statements}
\label{Sec:Introduction}

The science of architected materials is a branch of the science of synthetic materials that can be defined as the art of achieving full control over the classical degrees of freedom of the material and of their couplings. With such control, an architected material can be ``programmed'' to perform specific functions. The attribute ``architected'' mainly refers to the spatial arrangements of the building blocks of the material and, typically, these building blocks have adjustable internal structures to facilitate various degrees of programming. In an era dominated by information technology, the main applications of such materials are related to sensing, information storage and information processing. All these applications rely on specific dynamical features supported by a given architecture and a central problem in the field is to determine which new architectures bring genuine novelty. More specifically, given two architectures and the means to modify the internal structure of the resonators, one will like to know if there are dynamical effects that can be produced with one architecture but not with the other. This problem is equivalent to classifying the dynamical matrices that can be generated under various architectures. Its solution can be found in this work for a broad class of architected materials in the quadratic regime (as defined in section~\ref{Sec:QR}).

We build on two seminal papers of Bellissard \cite{Bellissard1986} and Kellendonk \cite{KellendonkRMP95} that pointed out that atomic systems have canonical underlying $C^\ast$-algebras where the dynamics of electrons can be formalized, analyzed and classified. This was an important development in condensed matter physics because it unleashed many of the techniques specific to operator algebras, such as $K$-theory and noncommutative geometry \cite{ConnesBook}, thus bringing much needed analytical tools to the difficult field of aperiodic solids \cite{Bellissard1995,Bellissard2003}. For example, the observed quantized plateaus in the integer quantum Hall effect and the so called bulk boundary-correspondence were rigorously explained using such tools \cite{BellissardJMP1994,KellendonkRMP2002} (see also \cite{BourneAHP2020,BourneJPA2018,MeslandCMP2022}, for more recent developments engaging groupoid frameworks). These ideas penetrated also the world of architected materials, where the dynamics is carried by classical degrees of freedom (see {\it e.g.} \cite{ApigoPRM2018,ApigoPRL2019,NiCP2019,ChengPRL2020,
ChenPRApplied2021,RosaNC2021,ChenNatComm2021}). The operator algebraic framework has been instrumental in these applications for understanding global aspects of the dynamics, establishing connections between various architectures and for predicting novel effects. In the present work, we focus on architected materials that are assembled from copies of a fixed part, called here the seed resonator. As opposed to the atomic systems, where the nuclei are featureless,\footnote{At the energy scales operated in condensed matter physics.} these resonators have shapes and orientations and the Bellissard-Kellendonk formalism, which engages only the group of pure translations, is no longer optimal for this context. The present work was motivated by the need of a more general formalism, where the orientations of the resonators are treated more effectively. As we shall see, the architected materials dealt with here can be formalized as point patterns inside the full group ${\rm Iso}(\EM^d)$ of Euclidean isometries. Besides being a non-abelian group, ${\rm Iso}(\EM^d)$ contains rotations and reflections that move points over large distances, hence, formalizing the mentioned materials presents conceptual as well as technical challenges. 

The mathematical aspects of patterns in non-abelian groups have been steadily attracted an interest from the mathematics community in the past few years \cite{BjorklundDM2018,BjorklundAIF2020,BjorklundAM2021,
BjorklundPLMS2018,BjorklundJFA2021,BjorklundMZ2022,
BeckusArxiv2021}. In fact, during the completion of our work, Enstad, Raum and van Velthoven posted the works \cite{Enstad1Arxiv2022,Enstad2Arxiv2022} on frame theory for patterns in locally compact second countable (lcsc) groups. These works not only cleared for us all the technical challenges, but they also supplied the natural mathematical framework we were looking for. Using this input together with empirical facts, for materials consisting of discrete architectures of self-coupling identical resonators, such as mechanical resonators or acoustic and photonic cavities, here are the statements we have to communicate:

\begin{enumerate}[\noindent 1)]
\item The architecture of such a material is specified by a uniformly separated closed subset $\Ll$ of ${\rm Iso}(\EM^d)$ (see Def.~\ref{Def:US}). \vspace{0.2cm}

\item Every such closed set has a canonically associated lcsc \'etale groupoid $\Gg_\Ll$ and a separable $C^\ast$-algebra $C^\ast_r(\Gg_\Ll)$ \cite{Enstad1Arxiv2022}. \vspace{0.2cm}

\item The dynamical matrix generating the dynamics of the material in the quadratic regime defines a linear operator over the Hilbert space $\ell^2(\Ll,\CM^N)$, where $N$ is the number of internal degrees of freedom carried by the seed resonator. \vspace{0.2cm}

\item If the physics involved in the coupling of the resonators is Galilean invariant, then any dynamical matrix can be generated from a left regular representation of $\KM \otimes C^\ast_r(\Gg_\Ll)$, where $\KM$ is the algebra of compact operators. \vspace{0.2cm}

\item By changing the internal structure of the seed resonator, the dynamical matrices sample the entire self-adjoint section of $\KM \otimes C^\ast_r(\Gg_\Ll)$. Thus, $\KM \otimes C^\ast_r(\Gg_\Ll)$ is the {\it smallest} $C^\ast$-algebra that can account for the dynamics of the material. It can be canonically associated to an architected material by both mathematical and physical means. \vspace{0.2cm}

\item The isomorphism classes of the canonically associated $C^\ast$-algebras sort the architected materials into classes of materials that are identical from dynamics point of view. \vspace{0.2cm}

\end{enumerate}

\begin{corollary} The research of this kind of architected materials essentially reduces to classifying the mentioned \'etale groupoids with respect to the equivalence relation introduced by Muhly, Renault, Williams \cite{MuhlyJOT1987}.
\end{corollary}

Furthermore, many dynamical effects in materials science are brought to life by deformations of materials, such as when one seeks topological pumping \cite{ChengPRL2020} or topological spectral flows \cite{Lux2Arxiv2022}. Many of these effects are predicted by means of topology and the biggest challenge for the field, in this respect, is the identification of the topological space in which the deformations of the materials occur. The stabilization of the groupoid $C^\ast$-algebra mentioned above gives the solution to this challenge if the deformations engage only the internal structures of the resonators or, more generally, if the deformations keep the material inside a fixed isomorphism class. If this is the case, then specialized tools developed for operator algebras can be employed to detect the dynamical effects that can be generated with an architected material (see section~\ref{Sec:Outlook}), generally with great effectiveness and precision. 

The paper is organized as it follows. Section~\ref{Sec:Emp} introduces the physical systems and specifies the assumptions made. They are by no means restrictive and they all can be straightforwardly incorporated in the designs of materials. For example, we argue that, to maintain control over the dynamics, the materials should  have finite coupling range and display certain continuity against deformations of their architectures. The section also explains how the configurations of these physical systems are connected to patterns inside ${\rm Iso}(\EM^d)$ and it derives the specific form of the dynamical matrices under the assumption of Galilean invariance, which represents the point of departure for the operator theoretic framework. The first part of section~\ref{Sec:Patt} reproduces from \cite{Enstad1Arxiv2022} the relevant facts about the topological space of closed subsets of lcsc groups and specific classes of patterns. It also introduces the related concepts of hull and transveral of a pattern. The second part of section~\ref{Sec:Patt} is devoted to examples, which, among other things, enable us to explain what has been gained by upgrading from the group of translations to the full group of Euclidean isometries. Section~\ref{Sec:DGM} introduces the canonical groupoid associated to a closed subset of a lcsc group and reproduces from \cite{Enstad1Arxiv2022} a fundamental result stating that this groupoid is \'etale if and only if the closed subset is uniformly separated. The rest of the section is devoted to the associated groupoid algebra and its relation to the space of dynamical matrices discussed in section~\ref{Sec:Emp}. It is in this part where the physical meaning of the groupoid $C^\ast$-algebra is established. The last section discusses possible research directions opened up by the formalism introduced by our work. Since we want to keep the message of this paper entirely focused on the synergy between the mathematical framework and the physics of architected materials, these research directions are only briefly mentioned and actual results along those lines will be presented elsewhere.

Lastly, we point out that there are several works that address point symmetries of tiled spaces (see {\it e.g.} \cite{RandPhDThesis,StarlingCMP2015} and references therein). There is, however, almost no overlap between these works and our, because the rotations considered in the former are emerging symmetries stemming from particular arrangements of featureless tiles ({\it i.e.} without internal structure). In our work, the rotations play an active role in the sense that they specify the internal structures the resonators of an ensemble, relative to that of a seed resonator. Patterns of resonators derived from actions of generic discrete groups on $\RM^d$ were considered by one of the authors in \cite{ProdanJGP2019}. This work also has little overlap with the present work because \cite{ProdanJGP2019} deals only with point-patterns of $\RM^d$ and the $C^\ast$-algebras of the dynamical matrices are crossed products rather than groupoid algebras. Furthermore, all these mentioned works are more about exploring specific interesting classes of materials and no attempts at a general classification are made.

\section{Empirical facts}\label{Sec:Emp}

This section supplies a precise description of the physical assumptions that make possible the theoretical predictions presented in the following sections. Along the way, this section lays out our recommendations for experimental physicists and engineers on how to build architected materials that fall under the umbrella of our framework. Regarding the exposition, our plan is to describe in details a rather specific class of physical systems, those of coupled mechanical resonators, and then supply a list other physical systems with similar attributes. 

\subsection{A laboratory setting}\label{Sec:Setting} A mechanical resonator is a confined mechanical system carrying a finite set $\bm q = \{q_1, \ldots, q_N\}$ of degrees of freedom and displaying a quadratic Lagrangean
\begin{equation}
L_0(\dot{\bm q},\bm q) = \tfrac{1}{2}\sum_{x \in \Ll} \left[\dot {\bm q}_x \cdot \widehat M_0 \cdot \dot {\bm q}_x^T - \bm q_x \cdot \widehat V_0 \cdot \bm q_x^T \right], \quad  \widehat M_0, \widehat V_0 \in M_N(\CM).
\end{equation}
Throughout, $M_N(\CM)$ denotes the algebra of $N \times N$ matrices with complex entries. The degrees of freedom will always be observed and quantified using equipment that is rigidly attached to the resonator or rather to the frame of the resonator. The equipment can be a sensing device or a simple video camera. Abstractly, this is conveyed by a co-moving frame and by a local research assistant that carries the observations and the measurements in such a rigidly attached frame. 

We will be concerned with finite and infinite clusters of identical resonators in the physical space $\EM^d$, $d=1,2,3$. Thus, apart from their locations and orientations, all resonators are identical to a {\it seed} resonator. In our framework, the identical resonators carry with them the same equipment, which is attached in exactly the same fashion on the frame of the resonator. As a consequence, the numerical values $\{q_1, \ldots, q_N\}$ recorded by the local research assistants are not affected by translations, rotations or reflections of the resonators. 

\begin{remark} {\rm The above details are quite central to our approach (see Remark~\ref{Re:CoMovF}) and should not be taken lightly. For example, what we are proposing here is quite different from the case of atomic systems, where the atomic orbitals are always quantified in a fixed laboratory frame, common to all atoms. 
}$\Diamond$
\end{remark}

The resonators come with their own force fields and a pair of resonators self-couple without any human intervention (see Examples~\ref{Ex:1} and \ref{Ex:2}). In abstract terms, this means that the Lagrangean $L$ of a pair or a cluster of resonators is entirely determined by the positions and orientations of the individual resonators. If $\Ll$ encodes the geometric configuration of a cluster of $r$ identical resonators, referred to as the architecture of the cluster, then the entire physics is encoded in the relation
\begin{equation}\label{Eq:L1}
\Ll \mapsto L_\Ll(\dot {\bm q}_1, \ldots, \dot {\bm q}_r,\bm q_1, \ldots, \bm q_r),  \quad \dot{\bm q}_i, \ \bm q_i \in \RM^N.
\end{equation}

\begin{remark} {\rm The relation~\eqref{Eq:L1} can be actually mapped experimentally (see {\it e.g.} \cite{ApigoPRM2018}) and this is highly recommended, especially if an experimental platform is used repeatedly. 
}$\Diamond$
\end{remark}

\begin{remark} {\rm Eq.~\eqref{Eq:L1} assumes that degrees of freedom belong to the specific manifold $\RM^N$. This is sufficient for our purposes, because we are interested only in the regime of small oscillations, where we can use a local chart of the configuration manifold of the seed resonator, to map its degrees of freedom into $\RM^N$.
}$\Diamond$
\end{remark}

\subsection{Formalizing the space of architectures}

 \begin{figure}[t]
\center
\includegraphics[width=0.9\textwidth]{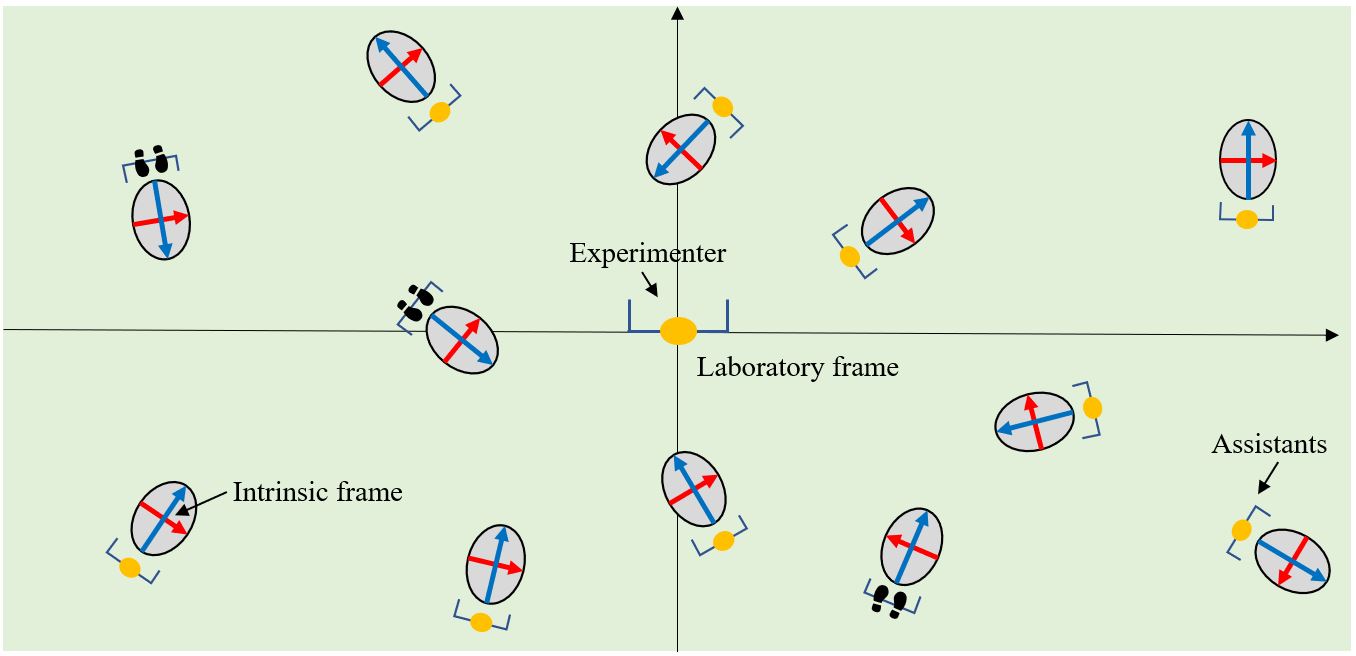}\\
  \caption{\small Example of a physical systems addressed in this work. It consist of identical self-coupling resonators placed at selected points of the plane and with selected orientations. Assistants observe and quantify the dynamics of the local degrees of freedom using a local frame rigidly attached to the resonators.
}
 \label{Fig:RL}
\end{figure}

Pictorially, the physical system we just described can be represented as in Fig.~\ref{Fig:RL}. We are going to exploit the fact that the resonators are identical, in which case the only distinction between two resonators in a cluster is reflected in the positioning of their intrinsic frames. If the laboratory frame is fixed once and for all, then the intrinsic frames of the resonators can be generated from the laboratory frame by applying a rotation followed by a translation. This simple line of reasoning reveals that the architecture $\Ll$ of a cluster of resonators is completely specified by a subset of the topological group ${\rm Iso}(\EM^d)$ of isometries. We will denote this subset by the same symbol $\Ll$. As indicated in Fig.~\ref{Fig:RL}, we include in our analysis proper as well as improper space transformations. 

Among other important things, this development enables us to label the resonators in a more natural way by elements of ${\rm Iso}(\EM^d)$. Specifically, we will write the degrees of freedom as $\{\bm q_x\}_{x \in \Ll}$. The right action of the group on itself will be indicated by a dot. Concretely,
\begin{equation}
{\rm Iso}(\EM^d) \ni g \mapsto g \cdot x : =  x g^{-1}, \quad \forall \  x \in {\rm Iso}(\EM^d).
\end{equation}
This action extends from points to subsets of ${\rm Iso}(\EM^d)$. Therefore, it makes sense to consider the action $g \cdot \Ll$ of the group on the architecture. The action of an isometry $g \in {\rm Iso}(\EM^d)$ on $p \in \EM^d$ will be indicated by the same symbol $g \cdot p$.

\begin{remark}{\rm The right action distinguishes itself from the left action of the group due to the following property: The configurations of the resonators labeled by $x$ and $g x$ drift apart from each other if $x$ wonders away from the neutral element and $g$ contains a rotation. This, however, is not the case if the right action of the group is used instead, as it can be easily seen by the rule of multiplication in ${\rm Iso}(\EM^d)$ shown in Appendix~\ref{Appendix}.}$\Diamond$
\end{remark}

\subsection{Galilean invariance} 

We now imagine an experimenter engaged in the process of mapping~\eqref{Eq:L1} and suppose the experimenter examines an architecture $\Ll'$ that is related to another architecture $\Ll$ by $\Ll' = g \cdot \Ll$, $g \in {\rm Iso}(\EM^d)$.  If the physics involved in the coupling of resonators is Galilean invariant, then the experimentally mapped Lagrangeans will enter the following relation
\begin{equation}\label{Eq:L2}
L_{\Ll'}\left(\{\dot{\bm q}_{g \cdot x}\}_{x \in \Ll},\{\bm q_{g \cdot x}\}_{x \in \Ll}\right) = L_\Ll \left(\{\dot{\bm q}_{g \cdot x}\}_{x \in \Ll},\{\bm q_{g \cdot x}\}_{x \in \Ll} \right).
\end{equation}

Note that the labeling of the resonators by group elements is quite instrumental here. Indeed, one should be aware that a linear indexing ({\it i.e.} some arbitrary ordering)  of the resonators cannot be made uniformly for all architectures $\Ll$, because certain deformations of $\Ll$ can switch two resonators. In such cases, it will become ambiguous which ordering to use. Quite the opposite, the indexing by the group elements fits naturally with the deformations of the architectures and this is why the Galilean invariance can be formulated as cleanly as in Eq.~\eqref{Eq:L2}. 

\subsection{Quadratic regime}\label{Sec:QR}  It is convenient to reference the Lagrangeans from the individual quadratic contributions of the uncoupled resonators. Hence, we write
\begin{equation}
L_\Ll = \tfrac{1}{2}\sum_{x \in \Ll} \left[\dot {\bm q}_x \cdot \widehat M_0 \cdot \dot {\bm q}_x^T - \bm q_x \cdot \widehat V_0 \cdot \bm q_x^T \right] - V_\Ll \big(\{\bm q_x\}_{x \in \Ll} \big),
\end{equation}
where $V_\Ll$ plays the role of the full coupling potential. In the regime of weak coupling, which we can be sure it occurs at least when the resonators are spread apart, the full potential $\sum_{x \in \Ll} \frac{1}{2} \bm q_x \cdot \widehat V_0 \cdot \bm q_x^T - V_\Ll \left(\{\bm q_x\}_{x \in \Ll} \right)$ has a unique minimum occurring at $\{\bar{\bm q}_x(\Ll)\}_{x \in \Ll}$ in the vicinity of $\{\bm q_x =0\}_{x \in \Ll}$. If we restrict the experimental observations to the regime of small oscillations around such minima, the Lagrangeans can be approximated by quadratic expressions
\begin{equation}
 L_\Ll = \tfrac{1}{2}\sum_{x \in \Ll}  \dot {\bm \xi}_x \cdot \widehat M_0 \cdot \dot {\bm \xi}_x^T  -  
\tfrac{1}{2} \sum_{x,x' \in \Ll}\, {\bm \xi}_{x} \cdot \widehat W_{x,x'}(\Ll) \cdot {\bm \xi}_{x'}^T ,
\end{equation} 
where $\bm \xi_x = \bm q_x - \bar{\bm q}_x$ and 
\begin{equation}
\left(\widehat W_{x,x'}(\Ll)\right) _{\alpha \beta} = \delta_{x,x'} \left( \widehat V_0\right)_{\alpha,\beta}+ \left . \frac{\partial^2 V_\Ll}{ \partial (\bm q_x)_\alpha \partial (\bm q_x')_\beta } \right |_{\bar{\bm q}(\Ll)} ,
\end{equation}
for $x,x'\in \Ll$. We will refer to these objects as the coupling matrices. 

A direct consequence of Eq.~\eqref{Eq:L2} is the following equivariance property of the coupling matrices:
\begin{equation}\label{Eq:EquiV1}
\widehat W_{g \cdot x,g \cdot x'} (g \cdot \Ll ) = \widehat W_{x,x'}(\Ll), \quad g \in {\rm Iso}(\EM^d),
\end{equation}
which holds whenever the physics involved in the coupling of the resonators is Galilean invariant. For example, this invariance will be broken in the presence of external fields or for couplings involving external frames and fixings.

\begin{remark}\label{Re:CoMovF}{\rm Our choice of observing and quantifying the degrees of freedom in a co-moving frame is prominently manifested in Eq.~\eqref{Eq:EquiV1}. Indeed, if instead the choice was to observe them from the fixed laboratory frame, then we had to specify how the degrees of freedom transform under the Euclidean transformations and the right side of Eq.~\eqref{Eq:EquiV1} would contain a conjugation by such transformation.
}$\Diamond$
\end{remark}

\begin{example}\label{Ex:1}{\rm The experimental platform introduced in \cite{ApigoPRM2018} uses a seed resonator that consists of a rigid body fitted with permanent magnets. The shape and mass distribution of the seed resonator as well as the positions of the permanent magnets define the internal structure. If the seed resonator aligned with the laboratory frame, hence the one labeled by the neutral element $e$, produces a magnetic field $\bm B_e(p)$ at the point $p \in \EM^3$, then the resonator labeled by $x=(t_x,\bm r_x) \in {\rm Iso}(\EM^3)$ (see Apendix~\ref{Appendix} for notation) produces at the same point $p$ a magnetic field 
\begin{equation}
\bm B_x(p)=\gamma(\bm r_x) \cdot \bm B_e(x^{-1} \cdot p),
\end{equation} 
where $\gamma$ represents the action of rotations on pseudo-vectors (or 2-forms). Now, let $\Ll$ be a finite point set in ${\rm Iso}(\EM^3)$ specifying a stable equilibrium configuration of an ensemble of such resonators. If $\tilde x$ is an out of equilibrium configuration of the resonator labeled by $x \in \Ll$, then we encode its internal degrees of freedom in the element $\bm q_x : = x^{-1} \tilde x \in {\rm Iso}(\EM^3)$.\footnote{We can view $\bm q_x$ as an element of $\RM^6$ by using a local chart of the Lie group ${\rm Iso}(\EM^d)$. Note that, in the actual experiments, $\bm q_x$ are constraint on sub-manifolds of ${\rm Iso}(\EM^3)$.} In the experimental conditions of \cite{ApigoPRM2018}, the potential energy of the mechanical system comes from the magnetic field's energy, written below in specific units,\footnote{The energy of an object carrying a saturated magnetization $\bm M$ in an external magnetic field $\bm B_{\rm ext}$ is $-\int_{\EM^3} {\rm} dp \, \langle \bm M(p),\bm B_{\rm ext}(p) \rangle$. Since the self-iteraction does not vary with the configuration of the object, $\bm B_{\rm ext}$ can promoted to the net magnetic field $\bm B_{\rm net}$. This together with the fact that $\int_{\EM^3} {\rm d} p \, \langle \bm B_{\rm net}(p) - 4 \pi \bm M(p),\bm B_{\rm net}(p) \rangle=0$, valid in the absence of free currents, gives Eq.~\eqref{Eq:ME} up to a multiplicative factor. The latter is absorbed in our chosen units for energy. } 
\begin{equation}\label{Eq:ME}
E = -\int_{\EM^3} {\rm d} p \, \langle \bm B_{\rm net}(p),\bm B_{\rm net}(p) \rangle, 
\end{equation} 
where $\langle \cdot,\cdot \rangle$ is the standard scalar product of $\RM^3$. The net magnetic field corresponding to a dynamical configuration ({\it i.e.} out of equilibrium) is a linear superposition 
\begin{equation}
\bm B_{\rm net}(p) =\sum_{x \in \Ll} \bm B_{x \bm q_x}(p)= \sum_{x \in \Ll} \gamma(\bm r_{x \bm q_x}) \cdot \bm B_e \big ( \bm q_x^{-1} x^{-1} \cdot p \big)
\end{equation} 
and the mechanical potential energy takes the form
\begin{equation}
V_\Ll  = -\int_{\EM^3} {\rm d} p \, \sum_{x,y \in \Ll} \big \langle \gamma(\bm r_{x \bm q_x}) \cdot \bm B_e( \bm q_x^{-1} x^{-1} \cdot p),\gamma(\bm r_{y \bm q_y}) \cdot \bm B_e( \bm q_y^{-1} y^{-1} \cdot p) \big \rangle.
\end{equation}
Using the invariance of the Lebesgue measure and of the scalar product against the actions of ${\rm Iso}(\EM^3)$, one can manually check the equivariance of the potential
\begin{equation}\label{Eq:L7}
V_{g \cdot \Ll}\left(\{{\bm q}_{g \cdot x}\}_{x \in \Ll}\right) = V_\Ll \left(\{\bm q_{g \cdot x}\}_{x \in \Ll} \right).
\end{equation}
In fact, using those simple facts, we can re-write the potential as
\begin{equation}
V_\Ll  = -\int_{\EM^3} {\rm d} p \, \sum_{x,y \in \Ll} \big \langle \gamma(\bm r_{\bm q_x}) \cdot \bm B_e(\bm q_x^{-1} \cdot p),\gamma(\bm r_{x^{-1} y \bm q_y }) \cdot \bm B_e( \bm q_y^{-1} y^{-1} x \cdot p) \big \rangle,
\end{equation}
hence making it manifestly equivariant.
}$\Diamond$
\end{example}

\begin{example}\label{Ex:2}{\rm The case of a seed resonator fitted with permanently electrically polarized components can be analyzed in a similar manner. Additional examples are discussed in section~\ref{Sec:OtherSys}.}$\Diamond$
\end{example}

\subsection{Dynamical matrix} From now on, we restrict the discussion to the quadratic regimes. Consider generalized driving forces 
\begin{equation}
\bm F_x(t)={\rm Re}\left[ (f_x^1,\ldots,f_x^N) e^{\imath \omega t}\right]={\rm Re} [ \bm f_x e^{\imath \omega t} ], \quad x \in \Ll,
\end{equation}
applied on each of the resonators. To describe the response of the cluster more effectively, we encode the driving forces in the vector $|\bm f \rangle = \sum_{x \in \Ll} \bm f_x \otimes |x\rangle$ from the Hilbert space $\CM^N \otimes \ell^2(\Ll)$. Up to a short transient motion, the system responds as 
\begin{equation}
\bm q_x(t) = \bar{\bm q}_x + \widehat M_0^{-\frac{1}{2}} \cdot {\rm Re}\left[  \bm \xi_x e^{\imath \omega t} \right],
\end{equation}
where $\bm \xi_x$ are the coefficients of the vector $|\bm \xi \rangle = \sum_{x \in \Ll} \bm \xi_x \otimes |x\rangle$ from $\CM^N \otimes \ell^2(\Ll)$ satisfying the equation
\begin{equation}\label{Eq:DynMat}
(D_\Ll - \omega^2\, I) |\bm \xi \rangle = \widehat M_0^{-\frac{1}{2}}\otimes I \, |\bm f \rangle, \quad  D_\Ll = \sum_{x,x' \in \Ll} w_{x,x'}(\Ll) \otimes |x\rangle \langle x'|,
\end{equation}
with
\begin{equation}\label{Eq:Ws}
w_{x,x'}(\Ll)=\widehat M_0^{-\frac{1}{2}} \, \widehat W_{x,x'}(\Ll) \, \widehat M_0^{-\frac{1}{2}} \in M_N(\CM).
\end{equation}
Since arbitrary time-dependent driving forces can be decomposed in Fourier components, the conclusion is that the response of the cluster to an external excitation reduces to the study of the spectral properties of the dynamical matrix $D_\Ll$, viewed as a linear operator over the Hilbert space $\CM^N \otimes \ell^2(\Ll)$.

\subsection{Controlling dynamics}\label{Sec:Contr} In the applications of architected materials, one needs control over the coupling matrices, that is, a researcher should be able to measure them with reasonable precision and be able to modify them according to the specifications of a given application. Here, we identify a class of dynamical matrices that have a positive chance to be controlled in a laboratory. Of course, these issues are relevant only for infinite architectures.

A coupling matrix $w_{x,x'}(\Ll)$ is said to have finite coupling range if it vanishes whenever $x' \cdot x$ is outside a fixed compact vicinity of the neutral element of ${\rm Iso}(\EM^d)$. We claim that coupling matrices that have finite range and depend continuously on $\Ll$ can be measured in a laboratory. Indeed, the equivariance relation~\eqref{Eq:EquiV1} assures us that it is enough to focus the observations only on $x,x'$ in a compact vicinity of the neutral element and, due to the assumed continuity, a good approximation of the map $\Ll \to w_{x,x'}(\Ll)$ can be derived by interpolating a finite sampling of $\Ll$'s. This assures us that, even for an infinite architecture, we still have a chance to measure with enough precision every single coupling matrix of the system. 

We can go one step beyond the class of systems described above. Indeed, any dynamical matrix that can be approximated by such dynamical matrices with finite coupling ranges can be also controlled in a laboratory. Of course, when mentioning approximations one needs to specify a norm on the dynamical matrices. The correct norm to consider here is the operator norm on $\CM^N \otimes \ell^2(\Ll)$, because convergence of dynamical matrices in this norm implies convergence of their resonant spectra. This gives us control over the spectral properties.

\subsection{A look ahead}\label{Sec:LAh} We encoded the dynamics of the clusters of resonators in the dynamical matrix $D_\Ll$. In the presence of Galilean invariance, the equivariance relation~\eqref{Eq:EquiV1} implies
\begin{equation}\label{Eq:EquiV2}
w_{g \cdot x,g \cdot x'} (g \cdot \Ll ) = w_{x,x'}(\Ll), \quad g \in {\rm Iso}(\EM^d),
\end{equation}
which can be used to reduce the expression of the dynamical matrix 
\begin{equation}\label{Eq:FinalD}
D_{\Ll} =\sum_{x,x' \in \Ll} w_{x' \cdot x,e}(x' \cdot \Ll) \otimes |x\rangle \langle x'|.
\end{equation}
Among other things, this shows that the entire dynamics is encoded in $M_N(\CM)$-valued map $w_{g,e} (\Ll)$, defined on tuples $(g,\Ll)$ with $g \in \Ll$. This together with the discussion in subsection~\ref{Sec:Contr} is the point of departure for the operator algebraic framework. Indeed, as we shall see in section~\ref{Sec:DGM}, any dynamical matrix as in Eq.~\eqref{Eq:FinalD} can be generated as the left-regular representation of an element from a stable, separable $C^\ast$-algebra canonically associated to $\Ll$. This is the stabilization of the groupoid $C^\ast$-algebra introduced in \cite{Enstad1Arxiv2022},  which generalizes Bellissard-Kellendonk construction from the context of the abelian group $\RM^d$ to that of general locally compact groups.  To formulate this $C^\ast$-algbera, we first need to give sense to the continuity of the map~\eqref{Eq:L1} and of the coupling matrices, that is, to specify natural topologies on the domains of these maps. In fact, the key to the operator theoretic formalism rests on the leap from the physical space to the topological space where the architectures $\Ll$ live. 

As already implied above, the $C^\ast$-algebra we just mentioned is invariant to modifications of the internal structure of the seed resonator. Furthermore, by changing this internal structure, one can modify the coupling matrices and explore the entire $C^\ast$-algebra. Thus, it is possible to identify this $C^\ast$-algebra with the smallest $C^\ast$-algebra that covers the dynamics of a lattice $\Ll$ of identical resonators with adjustable internal structure. This effort is justified because the representation, structure and K theories of this $C^\ast$-algebra filter out all dynamical effects achievable with such mechanical systems (see it at work in {\it e.g.} \cite{ApigoPRM2018}). Furthermore, upon deformations of the architecture $\Ll$, one can find many other mechanical systems that have isomorphic $C^\ast$-algebras and, as such, are equivalent from the dynamical point of view. The outstanding conclusion is that the world of architected materials can be divided in well-defined classes and, if we want to discover new dynamical effects in architected materials, the task is reduced to exploring different Morita equivalence classes of $C^\ast$-algebras.

\begin{remark}{\rm At the first sight, there might be an objection to our statement that we can actually achieve all available coupling matrices, given the fact that our exposition was carried in the weak coupling regime. Note, however, that the coupling matrices can be also manipulated via the inertia matrix $\widehat M_0$, which is not constrained in any way by the weak coupling assumption. In particular, large values of the coupling matrices~\eqref{Eq:Ws} can be generated by reducing the inertia of the resonators.
}$\Diamond$
\end{remark}  

\subsection{Other physical systems covered by analysis}\label{Sec:OtherSys} Mode-coupling theory is an effective theory applying to weakly coupled discrete resonators. It works in specific windows of frequencies, where only a finite number of resonant modes are active for the un-coupled resonators. Once these modes are identified, the theory assumes that the collective resonant modes can be approximated by linear combinations of the local resonant modes. In other words, the full Hilbert space of collective resonant modes is projected onto the subspace spanned by such linear combinations. This automatically leads to discrete dynamical matrices as in Eq.~\eqref{Eq:DynMat}. If the coupling of the resonators is mediated by a homogeneous medium, then the equivariance relation~\eqref{Eq:EquiV2} holds, hence, the dynamics of these systems is covered by our analysis. Concrete examples include coupled acoustic cavities \cite{ChenPRApplied2021}, patterning of surfaces \cite{WangSA2023}, coupled plate resonators \cite{Lux2Arxiv2022}, coupled electromagnetic cavities or photonic resonators \cite{ZilberbergNature2018}.

It is important, however, to stress that our analysis and predictions are not bound only to those systems. Discrete resonators can be coupled via man-designed bridges and other synthetic couplings. Our predictions still apply as long as these couplings are designed from the geometric information encoded in $\Ll$ and are insensitive to the rigid displacement and rotation of the coupled pairs. Thus, the task of an experimentalist will be to engineer the pairwise coupling matrices $w_{x,x'}(\Ll)$, by natural or synthetic means, such that they depend continuously on $x$ and $x'$, as well as on $\Ll$, and, in addition, satisfy the equivariance relation~\eqref{Eq:EquiV2}.

Given the discussion in the previous paragraph, we will go as far as saying that interesting laboratory models can be constructed from lattices of general locally compact groups (see \cite{Lux1Arxiv2022} for more details on such efforts). If one is interested in discovering new dynamical features in a systematic and controlled fashion, this is certainly a path worth exploring. The mathematical formalism of the next sections is developed for this very general context.

\section{The space of patterns and associated concepts}
\label{Sec:Patt}

The work of Enstad and Raum in \cite{Enstad1Arxiv2022} introduces a purely topological route to the $C^\ast$-algebra canonically associated to a pattern ({\it i.e.} closed subset)  in a lcsc group. Their approach is extremely general since it employs the standard Fell topology of the space of closed subsets and it relies solely on topological means to define uniformly discrete patterns. The alternative will be to endow the group with a metric and to generalize the induced Hausdorff metric on the space of compact subsets to a metric on the space of closed subsets, as it has been successfully done in the context of the abelian group $\RM^d$ \cite{KellendonkRMP95,ForrestAMS2002,LenzTheta2003}. However, it has been already recognized that those constructions lead exactly to the Fell topology \cite{BellissardPrivate}, which can  be seen from \cite{BenedettiBook}[Prop.~E.1.3] and \cite{ForrestAMS2002}[pg.~16]. For these reasons, the formalism put forward in \cite{Enstad1Arxiv2022} feels very natural and fits well with all the details of the applications we have in mind. Therefore, we decided to adopt it here and abandon our earlier efforts on alternative approaches. As an upshot, most of the technical results needed along the way have been already supplied in \cite{Enstad1Arxiv2022}, which means we can focus on physical interpretations and examples.

\subsection{Topologizing the space of patterns} For us, a pattern in a topological space is just a closed subset of that space. As we have seen in section~\ref{Sec:Emp}, the physical systems like the one illustrated in Fig.~\ref{Fig:RL} are associated to discrete subsets of the lcsc group ${\rm Iso}(\EM^d)$. Of course, the main interest is in unbounded subsets and this is the main source of difficulties. Thus, our first task is to describe the space of closed subsets of a lcsc group. Most of what we say is reproduced from \cite{FellPAMS1962}.

\begin{definition}[\cite{FellPAMS1962}] Let $X$ be any fixed topological space and let $\Cc(X)$ be the family of all closed subsets, hence including the void set $\emptyset$. For each compact subset $K$ of $X$ and each finite family $\Ff$ of nonvoid open subsets of $X$, let $U(K;\Ff)$ be the subset of $\Cc(X)$  consisting of all $Y$ such that: (i) $Y \cap K = \emptyset$, and (ii) $Y \cap F \neq \emptyset$ for each $F \in \Ff$. The Fell topology on $\Cc(X)$ is defined by the subbasis consisting of such $U(K;\Ff)$.
\end{definition} 

\begin{proposition}[\cite{FellPAMS1962}] If $X$ is locally compact, then the space $\Cc(X)$ of patterns in $X$ is automatically a compact Hausdorff space.
\end{proposition}

If $X$ is lcsc, then $\Cc(X)$ is second-countable too and its topology can be described by convergence of sequences. The convergence of sequences for the Fell topology can be described in the following concrete terms:

\begin{proposition}[\cite{BenedettiBook}]\label{Eq:Cr1} Let $C_n, C \in \Cc(X)$ for each $n \in \NM$. Then $C_n$ converges to $C$ in the Fell topology if and only if both of the following statements hold: (i) Whenever $x \in C$ then there exist $x_n \in C_N$ such that $x_n \to x$, and (ii) Whenever $(n_k)_{k \in \NM}$ is a subsequence of $\NM$ and $x_{n_k} \in C_{n_k}$ converges to $x \in X$, then $x \in C$.
\end{proposition}

As we already mentioned in section~\ref{Sec:Emp}, one of the tasks for us is to address the continuity of the map~\eqref{Eq:L1}. The practical criterion from Proposition~\ref{Eq:Cr1} reflects exactly what one will intuitively think of a continuous deformation of the pattern $\Ll$, namely, a continuous displacement of its points. Thus, Fell topology on the domain of the map~\eqref{Eq:L1} gives the answer to our first challenge.

\subsection{Classes of discrete patterns} From here on, we specialize the discussion to the case when $X$ is a lcsc topological group $G$. In such instances, $G$ plays both the passive role of hosting a pattern and the active role of space transformations. For the former, we will denote the elements of $G$ by $x,y, \ldots$, and for the latter we will use $g,g',\ldots$. We continue to denote  its neutral element by $e$. The right action of the group on itself can be elevated to a continuous $G$-action on $\Cc(G)$:
\begin{equation}
g \cdot C = \{x g^{-1}, \ x \in C\}, \quad g \in G, \quad C \in \Cc(G).
\end{equation}

The following it is claimed in \cite{Enstad1Arxiv2022} to be standard terminology for point sets in frame theory, though we were not able to identify the source of it. 

\begin{definition}[\cite{Enstad1Arxiv2022}] For $\Ll\in \Cc(G)$ and $S \subseteq G$, one says that $\Ll$ is
\begin{enumerate}[{\rm 1)}]
\item $S$-separated if $|\Ll \cap g \cdot S| \leq 1$ for all $g \in G$;
\item $S$-dense if $|\Ll \cap g \cdot S| \geq 1$ for all $g \in G$.
\end{enumerate}
\end{definition}

\begin{definition}[\cite{Enstad1Arxiv2022}]\label{Def:US} A subset $\Ll \subset G$ is called uniformly separated if there exists a non-empty open set $U \subseteq G$ such that $\Ll$ is $U$-separated. The set $\Ll$ is called uniformly dense if there exists a compact subset $K \subseteq G$ such that $\Ll$ is $K$-dense. If $\Ll$ is both separated and relatively dense, the it is called a Delone set.
\end{definition} 

From our discussion in section~\ref{Sec:Emp}, it is evident that, in such applications, we are dealing with uniformly separated subsets of ${\rm Iso}(\EM^d)$. For example, the resonators in Fig.~\ref{Fig:RL} cannot be overlapped and also the regime of small couplings was invoked in section~\ref{Sec:Emp}. Thus, in these applications, the origin of the seed resonator can be surrounded by an imaginary ball $B$ of the Euclidean space and these rigidly attached balls should not overlap when copies of the seed resonator are incorporated in the designs of the metamaterials. This imaginary ball then leads to the open set $U = B \times O(d)$ in ${\rm Iso}(\EM^d)$ and to a $U$-separated architecture $\Ll$. This kind of designs, however, do not exploit to the fullest what ${\rm Iso}(\EM^d)$ has to offer.  In principle, any open set $U \subset {\rm Iso}(\EM^d)$ can be used to shape the resonators and this gives us hopes there is plenty to be explored and discovered with these physical systems.

A Delone pattern is usually associated to the bulk of a material. However, a typical application of the operator theoretic methods is to the so called bulk-defect correspondence principle \cite{ProdanJPA2021} and a metamaterial with a defect is rather associated to uniformly separated patterns. Hence our focus from now on is on such patterns. The following statement gives a first characterization of these families:

\begin{proposition}[\cite{Enstad1Arxiv2022}] For any non-empty open set $U \subseteq G$, the set of $U$-separated sets is closed in $\Cc(G)$.
\end{proposition}

Note that, among other things, the above statement assures us that the limit of a sequence of $U$-separated closed subsets is also $U$-separated.

\subsection{The hull and transversal of a pattern} The concepts introduced here make sense for generic patterns in a lcsc group $G$. Therefore, we will keep the presentation at that level of generality and then follow with specific examples relevant for the physical systems discussed in section~\ref{Sec:Emp}.

\begin{definition} The closure of the orbit of $\Ll \in \Cc(G)$ under the natural $G$-action is called the hull of $\Ll$:
\begin{equation}
\Omega_{\Ll} = \overline{\{g \cdot\Ll, \ g \in G\}}.
\end{equation}
\end{definition}

Since $\Cc(G)$ is compact, $\Omega_\Ll$ is compact for any $\Ll \in \Cc(G)$.  Furthermore, since the property of a closed subset of being $U$-separated is unaffected by translations, any member of $\Omega_\Ll$ is $U$-separated if $\Ll$ is $U$-separated. A fine point noted in \cite{Enstad1Arxiv2022} is that, if a pattern is not uniformly dense, then $\Omega_\Ll$ contains the empty set $\emptyset$ as an invariant point. Thus, the proper space for analysis is:

\begin{definition}[\cite{Enstad1Arxiv2022}] The punctured hull of $\Ll$ is defined as:
\begin{equation}
\Omega^\times_{\Ll} = \Omega_\Ll \setminus \emptyset.
\end{equation}
\end{definition}

$\Omega^\times_{\Ll}$ is a locally compact space. Most that will be said from now on builds on the dynamical system $(G,\Omega^\times_\Ll)$. An essential role will be undertaken by its standard Poincar\'e transversal:

\begin{definition}\label{Def:ST} The standard transversal of $(G,\Omega^\times_\Ll)$ is 
\begin{equation}
\Xi_\Ll = \{ \Ll' \in \Omega_\Ll, \ e \in \Ll'\} \subset \Omega_\Ll.
\end{equation}
\end{definition}

The standard transversal is closed in $\Omega_\Ll$ and has no intersection with $\emptyset$, hence it is actually a compact subspace of $\Omega^\times_\Ll$. The applications of the operator theoretic approach to materials science rests quite heavily on the ability to calculate $\Xi_\Ll$. In the physics literature, this space is called the phason space of the pattern. 

\begin{figure}[t]
\center
\includegraphics[width=\textwidth]{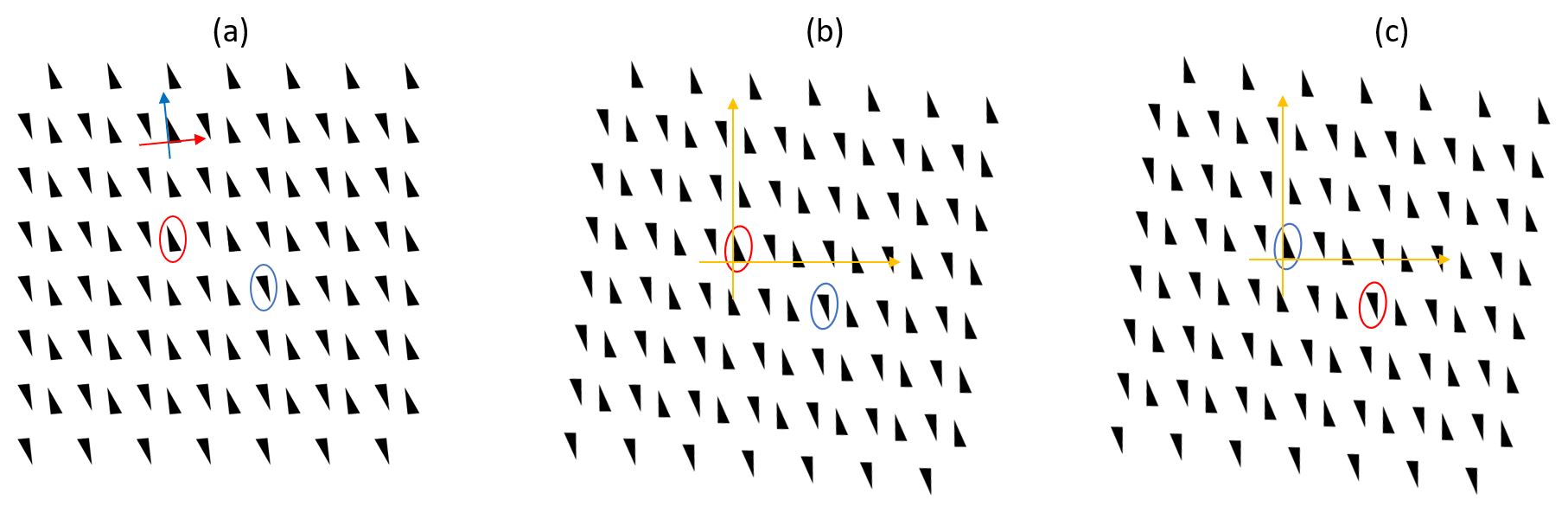}\\
  \caption{\small (a) Section of an infinite pattern generated with p2 wallpaper group. The intrinsic frame attached to the resonators is indicated by the blue/red axes. (b) A view of the pattern after it was rigidly handled until the resonator highlighted in red in panel (a) is aligned with the laboratory's coordinate system. (c) Same as (b) for the resonator highlighted in blue in panel (a). This should convince the reader that the pattern looks the same no matter which resonator we pick. As a result, the transversal of this pattern consists of a single point.
}
 \label{Fig:P2}
\end{figure}

\begin{figure}[b]
\center
\includegraphics[width=\textwidth]{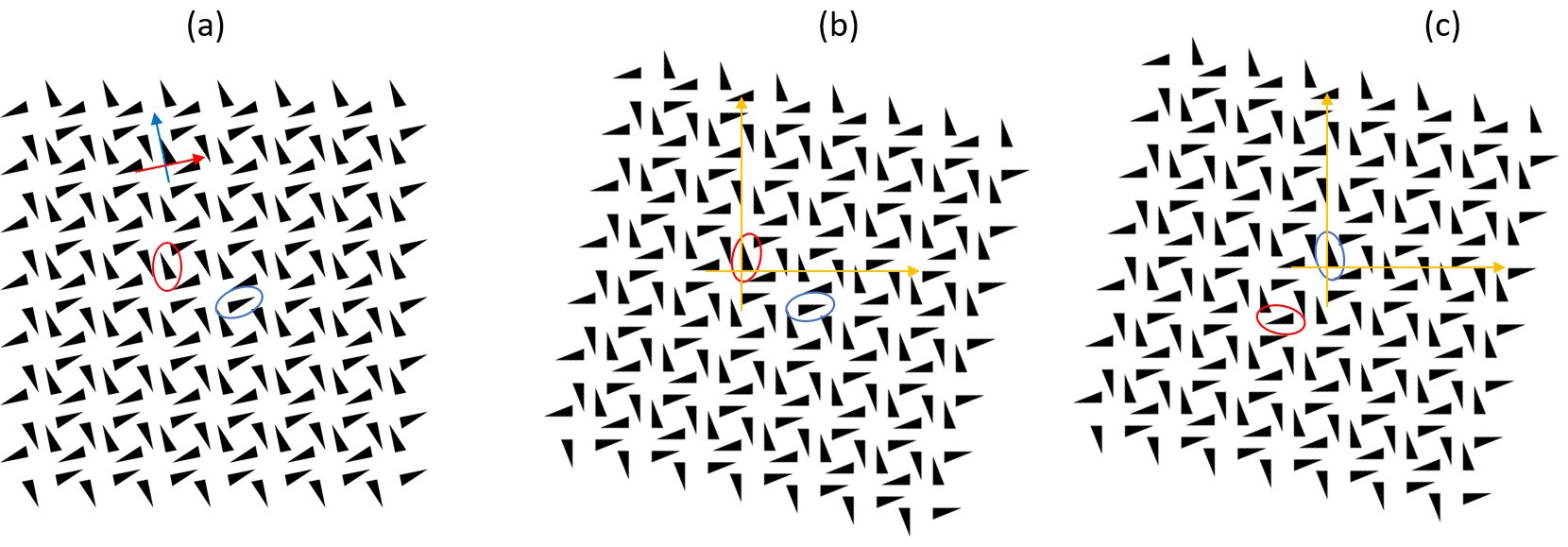}\\
  \caption{\small Same as Fig.~\ref{Fig:P2}, but for the wallpaper group p4.
}
 \label{Fig:P4}
\end{figure}

\begin{remark}\label{Re:Transversal}{\rm For the patterns described in section~\ref{Sec:Emp}, a practical way to think about the transversal is as follows. Using rigid motions, we can place a resonator of the pattern at the origin, with its intrinsic frame aligned with that of the laboratory. We take a picture of the resulted pattern with a camera located above the origin of the laboratory frame and we place that picture in $\Xi_{\Ll}$. We repeat the process until each resonator gets to sit at the origin once and then we take the topological closure of the discrete set of patterns so obtained. The criterion spelled in Proposition~\ref{Eq:Cr1} is instrumental for this last step.
}$\Diamond$
\end{remark}

\subsection{Examples}\label{Sec:Example} The success of the $C^\ast$-algebraic framework, as applied to the context discussed in section~\ref{Sec:Emp}, is very much conditioned by ones ability to map out the standard transversals. While these applied aspects will be scrutinized in our future works, here we illustrate several instances where the standard transversal can be mapped without much effort. We use these examples to highlight the differences emerging when the patterns are seen as subsets of ${\rm Iso}(\EM^d)$ or of $\RM^d$.

\begin{example}\label{Ex:P2}{\rm The pattern of resonators in Fig.~\ref{Fig:P2} was generated by acting with the space-group p2 on an initial triangle slightly displaced from the origin. This displacement is needed in order for the centers of the resonators to generate a uniformly separated point set. According to Remark~\ref{Re:Transversal}, we should move each resonator to the origin of the laboratory frame and align the two. This is done in Figs.~\ref{Fig:P2}(b,c) for two randomly picked resonators. As one can see, the resulted patterns look the same in the two situations, and evidently will look the same if we sample other resonators. As such, the standard transversal consists of a single point, the pattern that we see either in panel (b) or (c).
}$\Diamond$
\end{example}

\begin{example}\label{Ex:P4}{\rm The pattern from Fig.~\ref{Fig:P4} was generated as in Example~\ref{Ex:P2}, but using the wallpaper group p4. The conclusion is the same, namely the transversal of the pattern consists of a single point. These two findings are not singular. In fact, any pattern generated by acting with a wallpaper group on a seed shape  will have a singleton tranversal. Examples of architected materials designed with such algorithms can be found in \cite{Lux2Arxiv2022}.
}$\Diamond$
\end{example}

\begin{figure}[t]
\center
\includegraphics[width=\textwidth]{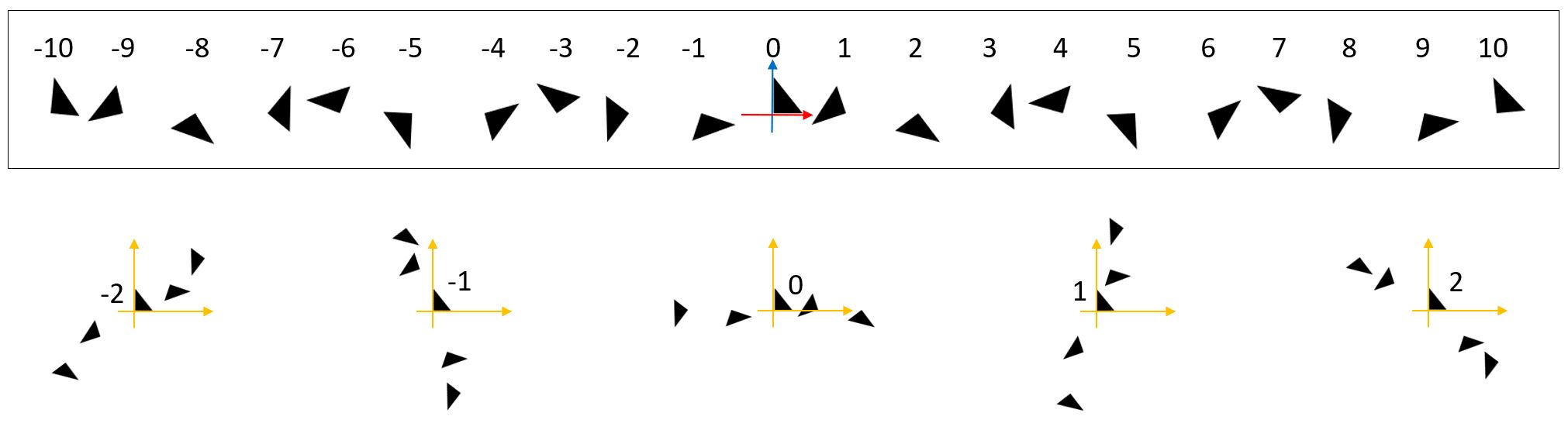}\\
  \caption{\small The pattern shown in the top row was generated by repeated translation and rotation by angle $\theta = \frac{2 \pi}{\sqrt{11}}$, starting from the seed triangle 0. The bottom row shows how the pattern looks after different resonators have been aligned with the laboratory frame.
}
 \label{Fig:1DQP}
\end{figure}

\begin{remark}\label{Re:GvsZ}{\rm As abstract patterns, the motifs from Figs.~\ref{Fig:P2} and \ref{Fig:P2} can be certainly seen as closed subsets of the group $\RM^2$ of pure translations. In this case, their standard transversals, as defined in Def.~\ref{Def:ST},  are homeomorphic to two, respectively, four copies of the seeding shape. If, instead, we follow Kellendonk's recipe for constructing a transversal and take the origin of the local frames as the puctures of the tiles, we find transversals that consists of two, respectively, four points. }$\Diamond$
\end{remark}

\begin{figure}[b]
\center
\includegraphics[width=\textwidth]{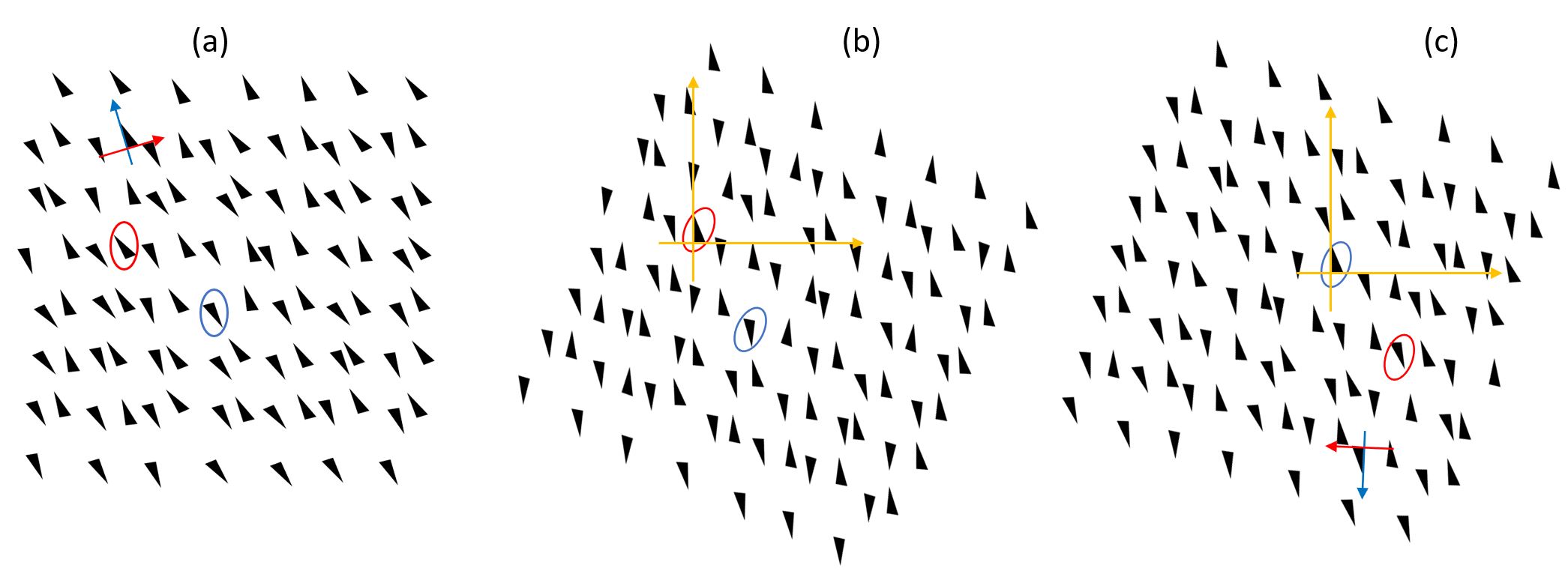}\\
  \caption{\small Same as Fig.~\ref{Fig:P2} but with random fluctuations in the configuration of the initial shape.
}
 \label{Fig:P2Disorder}
\end{figure}

\begin{example}\label{Ex:1DQP}{\rm Fig.~\ref{Fig:1DQP} displays a pattern of resonators generated by repeated application of $(t,\bm r) \in {\rm Iso}(\EM^2)$ on a seed triangle, where $t$ is a translation and $\bm r$ is a proper rotation by an angle $\theta$ that is incommensurate with $2\pi$ (see appendix~\ref{Appendix} for notation and other details). Fig.~\ref{Fig:1DQP} also shows rigid shifts of the pattern resulted from aligning different resonators with the laboratory's frame. Each of those patterns represent a point of the transversal $\Xi_{\Ll}$. As one can see, the axis of the pattern rotates and also the local environment of the aligned resonator changes as well as a result of these actions. Still, there is a bijection between the orbit of $\Ll$ generated by these alignments and the points $n \theta$ of the circle $2 \pi \RM/\ZM$. The latter densely fill the circle and it is not difficult to show (see \cite{KellendonkAHP2019} for such an exercise) that the transversal, which is the closure of the orbit, is homeomorphic with the full circle. 
}$\Diamond$
\end{example}

\begin{remark}{\rm Example~\ref{Ex:1DQP} can be viewed as a quasi 1-dimensional physical system without a defect, but note that its natural description and computation of the transversal still involves the Euclidean group ${\rm Iso}(\EM^2)$. As a point set of  ${\rm Iso}(\EM^2)$, this pattern is uniformly separated but not a Delone set. This also shows why the class of uniformly separated patterns is actually the natural one to consider for this kind of applications.
}$\Diamond$
\end{remark}

\begin{remark}{\rm We can use Example~\ref{Ex:1DQP} to showcase some practical advantages brought in by the leap to the full Euclidean group. The quasi-periodic patterns of point resonators are quite popular with the materials scientists. In the existing applications, some already mentioned in section~\ref{Sec:Introduction}, the patterns were generated by pure translations of the seed resonators, following certain algorithms (see {\it e.g.} \cite{ApigoPRM2018} or \cite{ChenPRApplied2021}). In such cases, the physical systems are treated as point sets in $\RM^d$ and the standard transversal can be shaped as circles or tori. Many dynamical effects can be, theoretically, brought to life by driving the phason along topologically nontrivial paths inside such transversals. This, unfortunately, is a very difficult task in practice, especially if only translations are allowed. For the pattern from Example~\ref{Ex:1DQP}, however, a driving of the phason can be accomplished by synchronously rotating the resonator while holding their centers fixed, which is a much simpler procedure. This is not an isolated example since upgrading from the group of pure rotations to the full group of Euclidean isometries brings in more opportunities that can be exploited in specific applications (see \cite{MartinezArxiv2023} for specific examples).
}$\Diamond$
\end{remark}

\begin{remark}{\rm It will be difficult here to apply Kellendonk's prescription, which requires tiling of the space.}$\Diamond$
\end{remark}

\begin{example}\label{Ex:P2Disorder}{\rm The pattern in Fig.~\ref{Fig:P2Disorder} was generated using again the discrete subgroup p2, but each isometry from the p2 group was applied on a different initial shape. These initial shapes were generated by applying $(t,\bm r) \in {\rm Iso}(\EM^2)$ on a triangle aligned with the laboratory frame, where $(t,\bm r)$ were sampled randomly and uniformly from the set $(t_0,\bm r_0) \cdot B_\epsilon \times [-\epsilon,\epsilon]$, where $B_\epsilon$ is the $\epsilon$-ball of $\EM^2$ centered at the origin. Here, we used the parametrization~\eqref{Eq:E2Topo} of ${\rm Iso}(\EM^2)$. These random fluctuations can be a result of inherent experimental errors. In his case, the transversal is the product space $\left( B_\epsilon \times [-\epsilon,\epsilon] \right)^{\times \rm p2}$ equipped with the product topology. This space comes with a continuous action of p2.
}$\Diamond$
\end{example}

\begin{example}{\rm Additional examples can be found in \cite{BeckusArxiv2021}.
}$\Diamond$
\end{example}

\section{Dynamics by groupoid methods}\label{Sec:DGM}

This section describes the groupoid canonically associated to a pattern. We pay equal attention to the algebraic structure and the topology of this groupoid. A Haar system is generally needed to build a $C^\ast$-algebra over such groupoid. An important aspect of the present context is that the groupoids associated to uniformly separated patterns are etal\'e and such groupoids come with an intrinsic Haar system, hence with an intrinsic $C^\ast$-algebra. The main statement of the section is that this $C^\ast$-algebra generates the dynamical matrices of the resonant physical systems discussed in section~\ref{Sec:Emp}, through its standard left regular representation.

\subsection{Topological groupoid associated to a pattern} To keep the exposition as streamlined as possible, we assume a readership familiar with basic notions related to groupoids. Valuable sources of information on the subject are \cite{RenaultBook,SimsSzaboWilliamsBook2020,WilliamsBook2}.

Any  topological dynamical system, hence $\left({\rm Iso}(\EM^d), \Omega^\times_{\Ll}\right)$ too, has a canonically associated transformation groupoid. A more general yet simpler structure is that of generalized equivalence and its associated topological groupoid:

\begin{definition}[\cite{WilliamsBook2},~p.~5]\label{Def:GGenEquiv} Let $G$ be a group and $X$ a set. Then $X \times G \times X$ can be given the structure of a groupoid by adopting the inversion function 
\begin{equation}
(x,g,y)^{-1} = (y,g^{-1},x),
\end{equation}
the range and source maps
\begin{equation}
\mathfrak{r}(x,g,y)=(x,e,x),\quad \mathfrak{s}(x,g,y)=(y,e,y),
\end{equation}
and by declaring that two triples $(x,g,y)$ and $(w,h,z)$ are composible if $y = w$, in which case the product is
\begin{equation}
(x,g,y)\cdot (y,h,z) : = (x,gh,z).
\end{equation} 
If $X$ is a topological space and $G$ is a topological group, then $X \times G \times X$ becomes a topological groupoid if endowed with the product topology.
\end{definition}

Using this simple groupoid structure, we can effortlessly  describe both the algebraic structure and topology of a transformation groupoid:

\begin{definition}[\cite{WilliamsBook2}~p.6] Let $G$ be lcsc and $X$ be a lcsc topological $G$-space. Then the transformation groupoid of $(X,G)$ is the lcsc topological sub-groupoid of the groupoid introduced in Definition~\ref{Def:GGenEquiv} consisting of the triples $(x,g,y)$ with $x=g\cdot y$. The topology is that induced from $X \times G \times X$.
\end{definition}

From now on, we will fix a pattern $\Ll_0 \in \Cc(G)$ in a lcsc group $G$. For the start, $\Ll_0$ can be any closed subset. When applied to $\left(G,\Omega^\times_{\Ll_0}\right)$, the above definition leads to the lcsc groupoid $\widetilde \Gg_{\Ll_0}$ with source $\tilde{\mathfrak s}$ and range $\tilde{\mathfrak r}$ maps
\begin{equation}
\tilde{\mathfrak s}\left(g \cdot \Ll,g, \Ll \right) = (\Ll,e,\Ll),  \quad \tilde{\mathfrak r}\left(g \cdot \Ll,g, \Ll \right) = \left(g \cdot \Ll, e , g \cdot \Ll \right).
\end{equation}
Its space of units $\widetilde \Gg_{\Ll_0}^{0} : = \tilde{\mathfrak s}(\widetilde \Gg_{\Ll_0}) = \tilde{\mathfrak r}(\widetilde \Gg_{\Ll_0})$ coincides with the punctured hull $\Omega^\times_{\Ll_0}$.

\begin{definition} The canonical topological groupoid $\Gg_{\Ll_0}$ associated to the closed pattern $\Ll_0 \in \Cc(G)$ is the restriction of $\widetilde \Gg_{\Ll_0}$ to the transversal $\Xi_{\Ll_0} \subset \widetilde \Gg_{\Ll_0}^{0}$ of Definition \ref{Def:ST}:
\begin{equation}\label{Eq:G0}
\Gg_{\Ll_0} : = \tilde {\mathfrak s}^{-1}(\Xi_{\Ll_0}) \cap \tilde{\mathfrak r}^{-1}(\Xi_{\Ll_0}).
\end{equation}
\end{definition}

Its space of units $\Gg_{\Ll}^{0}$ coincides with $\Xi_{\Ll_0}$, hence it is a compact space. Among other things, this assures us that its groupoid $C^\ast$-algebra has a unit. The canonical groupoid of a pattern can be characterized more concretely in a form that brings it more closely to the ones seen in the physics applications: 

\begin{proposition} The topological groupoid canonically associated to a pattern $\Ll_0 \in \Cc(G)$ consists of
\begin{enumerate}[\ \rm 1.]
\item The set of tuples
\begin{equation}
\Gg_{\Ll_0} = \left\{ (g, \Ll), \ \Ll \in \Xi_{\Ll_0}, \ g \in \Ll \right\} \subset G \times \Xi_{\Ll_0},
\end{equation}
equipped with the inversion map
\begin{equation}
(g,\Ll)^{-1} = (g^{-1},g \cdot \Ll) = g \cdot (e, \Ll)
\end{equation}
and lcsc topology inherited from $G \times \Xi_{\Ll_0}$.
\item The subset of $\Gg_{\Ll_0} \times \Gg_{\Ll_0}$ of composable elements
\begin{equation}
 \Gg_{\Ll_0}^{2} =\left\{\left((g',\Ll'),(g,\Ll)\right) \in \Gg_{\Ll_0} \times \Gg_{\Ll_0}, \  \Ll'=g\cdot \Ll \right\}
 \end{equation}
equipped with the composition 
\begin{equation}
(g',\Ll') \cdot (g,\Ll) = (g' g, \Ll).
\end{equation}
\end{enumerate} 
The source and range maps of $\Gg_{\Ll_0}$ are
\begin{equation}\label{Eq:RS}
\mathfrak s(g,\Ll) = (e,\Ll), \quad  \mathfrak r(g,\Ll) =  \left(e,g \cdot \Ll \right).
\end{equation}
\end{proposition}

\begin{proof} Since
\begin{equation}
\tilde{\mathfrak s}^{-1}(\Ll,e,\Ll) = \left\{(g \cdot \Ll,g, \Ll ), \ g \in G \right\}
\end{equation}
and
\begin{equation}
 \quad \tilde{\mathfrak r}^{-1}(\Ll,e,\Ll) = \left\{ (\Ll,g^{-1},g \cdot \Ll), g \in G \right\},
\end{equation}
$(g \cdot \Ll,g, \Ll )$ belongs to the intersection~\eqref{Eq:G0} if and only if $\Ll$ and $g \cdot \Ll$ both belong to $\Xi_{\Ll_0}$. This automatically constrains $g$ to be on the lattice $\Ll$ and, vice versa, if $g \in \Ll$, then $(g \cdot \Ll,g, \Ll )$ belongs to the intersection~\eqref{Eq:G0}. Thus $\Gg_{\Ll_0}$ can be  presented as the sub-groupoid of $\widetilde \Gg_{\Ll_0}$ consisting of the triples
\begin{equation}
(g \cdot \Ll,g, \Ll ) \in \widetilde \Gg_{\Ll_0}, \quad \Ll \in \Xi_{\Ll_0} \subset \Omega^\times_{\Ll_0}, \quad g \in \Ll,
\end{equation}
endowed with the topology inherited from $\widetilde G_{\Ll_0}$. Lastly, we can drop the redundant first entry of the triples.
\end{proof}

\begin{example}\label{Ex:RG}{\rm In the case when $\Xi_{\Ll_0}$ is generated with a wallpaper group, as in Examples~\ref{Ex:P2} and \ref{Ex:P4}, the transversal consists of single point and $\Gg_{\Ll_0}$ reduces to the wallpaper group itself.
}$\Diamond$
\end{example}

\begin{example}{\rm  For the setting from Example~\ref{Ex:P2Disorder}, $\Gg_{\Ll_0}$ is isomorphic with the transformation groupoid associated to $\left({\rm p2},\left( B_\epsilon \times [-\epsilon,\epsilon] \right)^{\times \rm p2}\right)$.
}$\Diamond$
\end{example}

In \cite{MuhlyJOT1987}, the authors introduced an  equivalence relation on groupoids that ensures that their associated $C^\ast$-algebras are Morita equivalent. In this context, they introduced the notion of abstract transversal:

\begin{definition}[\cite{MuhlyJOT1987}] Let $\Gg$ be a lcsc groupoid with source and range maps $\mathfrak s$ and $\mathfrak r$, respectively. An abstract transversal for $\Gg$ is a closed subset $\Xi$ of its space of units $\Gg^0$ such that $\Xi \cap \mathfrak r\left(\mathfrak s^{-1}(x)\right) \neq \emptyset$ for any $x \in \Gg^0$ and $\mathfrak s$ and $\mathfrak r$ are both open maps when restricted to $\mathfrak r^{-1}(\Xi)$. 
\end{definition}

\begin{remark}{\rm The notion of group action on topological spaces can be generalized to actions of groupoids \cite{MuhlyJOT1987} (see also \cite{WilliamsBook2}[Def.~2.1]). For the transformation groupoid associated to a dynamical system $(G,\Omega)$, the groupoid action on its space of units is determined by the action of $G$ on $\Omega$. Note that $\mathfrak r\left(\mathfrak s^{-1}(x)\right)$ appearing in the above definition is just the orbit of $x$ under the $G$-action. Thus, an abstract transversal intersects every $G$-orbit. The other conditions ensure that the notion of transversal intersection is compatible with the topological groupoid structure. Thus, $\Xi$ is the analog of a Poincar\'e section.
}$\Diamond$
\end{remark}

\begin{proposition}[\cite{MuhlyJOT1987}] If $\Xi$ is an abstract transversal for $\Gg$, then $\Gg$ and $\Gg |_\Xi : = \mathfrak s^{-1}(\Xi) \cap \mathfrak r^{-1}(\Xi)$ are equivalent groupoids. In particular, their respective groupoid $C^\ast$-algebras are Morita equivalent.
\end{proposition}

In our context, the following statement assures us that $\widetilde \Gg_{\Ll_0}$ and $\Gg_{\Ll_0}$ are equivalent groupoids.

\begin{proposition}[\cite{Enstad1Arxiv2022}] Let $\Ll_0$ be a uniformly separated pattern in $G$. Then the standard transversal $\Xi_{\Ll_0}$ from Definition~\ref{Def:ST} is an abstract transversal for $\widetilde \Gg_{\Ll_0}$.
\end{proposition}

\begin{remark}{\rm Both groupoids $\widetilde \Gg_{\Ll_0}$ and $\Gg_{\Ll_0}$ are equally useful in physical applications \cite{Bellissard1986,KellendonkRMP95}. For example, most of the time, the empirical input, such as probability measures on the spaces of units, is readily available for $\widetilde \Gg_{\Ll_0}$, but, as explained below, the analysis is simpler on $\Gg_{\Ll_0}$.}$\Diamond$
\end{remark}

\begin{definition}[\cite{SimsSzaboWilliamsBook2020},~Def.~8.4.1] A topological groupoid $\Gg$ is called \'etale if the range map is a local homeomorphism.
\end{definition}

\begin{remark}{\rm The \'etale groupoids can be thought as generalizations of discrete groups. Their fibers  $\mathfrak r^{-1}(x)$ and $\mathfrak s^{-1}(x)$ are discrete topological spaces for any $x \in \Gg^0$ and any \'etale groupoid admits a canonical Haar system supplied by the system of counting measures on each fiber \cite[Def. 2.6, Lemma 2.7]{RenaultBook}. }$\Diamond$
\end{remark}

The following remarkable result from \cite{Enstad1Arxiv2022} characterizes the uniformly separated patterns in a lcsc group via their canonical groupoid $C^\ast$-algebras:

\begin{proposition}[\cite{Enstad1Arxiv2022}] Consider a lcsc group $G$ and $\Ll_0 \in \Cc(G)$. Then the groupoid $\Gg_{\Ll_0}$ is \'etale if and only if $\Ll_0$ is uniformly separated.
\end{proposition}

The important conclusion is that any of the physical system described in section~\ref{Sec:Emp} has an underlying lcsc \'etale groupoid. 

\subsection{The groupoid $C^\ast$-algebra of a uniforlmy separated pattern}

The linear space $C_c\left(\Gg_{\Ll_0}, M_N(\CM)\right)$ of compactly supported $M_N(\CM)$-valued continuous functions on $\Gg_{\Ll_0}$ can be endowed with the associative multiplication
\begin{equation}\label{Eq:Conv}
(f_1 \ast f_2)(g,\Ll) =\sum_{g' \in \Ll}f_{1}\left(g' \cdot(g,\Ll)\right)  \cdot f_{2} (g', \Ll )
\end{equation}
and $\ast$-operation
\begin{equation}\label{Eq:Inv}
f^\ast(g, \Ll) = f\left((g,\Ll)^{-1}\right)^\dagger=f(g^{-1},g\cdot \Ll)^\dagger.
\end{equation}
The resulting $\ast$-algebra is completed to a $C^\ast$-algebra by using the natural right-module structure of $C_c\left(\Gg_{\Ll_0},M_N(\CM)\right)$ over the $C^\ast$-algebra $C\left (\Xi_{\Ll_0},M_N(\CM)\right)$. The latter can be promoted to a $C^\ast$-Hilbert module structure via the $C\left(\Xi_{\Ll_0},M_N(\CM)\right)$-valued inner product
\begin{equation}\label{Eq:InnerProd1}
\langle f_1|f_2 \rangle(\Ll) :=\rho(f_1^\ast \ast f_2)(\Ll) = \sum_{g \in \Ll} f_1(g,\Ll)^{\dagger} \cdot f_2(g,\Ll),
\end{equation}
where $\rho$ is the restriction map 
\begin{equation}
\rho :C_c\left(\Gg_{\Ll_0},M_N(\CM)\right) \rightarrow C \left(\Xi_{\Ll_0},M_N(\CM)\right), \ \ (\rho f)(\Ll) = f(e,\Ll).
\end{equation}
If $E_{\Ll_0}$ is the Hilbert $C^\ast$-module completion of $C_c\left (\Gg_{\Ll_0},M_N(\CM)\right)$ in this inner product, then the convolution from the left supplies a left action of the $\ast$-algebra $C_c\left(\Gg_{\Ll_0},M_N(\CM)\right)$ on $E_{\Ll_0}$ by bounded adjointable endomorphisms, extending the action of $C_c\left(\Gg_{\Ll_0},M_N(\CM)\right)$ on itself.
\begin{definition}[\cite{KhoshkamJRAM2002}]\label{Def:RecCStar}
The reduced groupoid $C^\ast$-algebra of $\Gg_{\Ll_0}$ over the $C^\ast$-algebra $M_N(\CM)$ is the completion of the core algebra $C_c\left(\Gg_{\Ll_0},M_N(\CM)\right)$ in the norm inherited from the embedding 
$$
C_c\left(\Gg_{\Ll_0},M_N(\CM)\right) \rightarrowtail {\rm End}^\ast(E_{\Ll_0})
$$ 
and is denoted $C^\ast_r \left(\Gg_{\Ll_0},M_N(\CM)\right)$.
\end{definition}

\begin{remark}\label{Re:Iso7}{\rm If $N=1$, we simplify the notation to $C^\ast_r (\Gg_{\Ll_0})$ and refer to this algebra as {\it the} groupoid $C^\ast$-algebra of $\Gg_{\Ll_0}$. There is a standard isomorphism $C^\ast_r \left(\Gg_{\Ll_0},M_N(\CM)\right) \simeq M_N(\CM) \otimes C^\ast_r (\Gg_{\Ll_0})$, so those more general $C^\ast$-algebras appearing in Definition~\ref{Def:RecCStar} are just matrix amplifications of the groupoid $C^\ast$-algebra of $\Gg_{\Ll_0}$. We recall that such matrix amplifications have no impact on the K-theory. The isomorphism we just mentioned will come handy in section~\ref{Sec:SortingD}, where we make the connection with the models discussed in section~\ref{Sec:Emp}. }
$\Diamond$
\end{remark}

\begin{remark}{\rm For a second countable, Hausdorff, \'etale groupoid $\Gg$, the reduced $C^\ast$-algebra is separable. In turn, this implies that its $K$-groups are countable \cite{BlackadarBook1998} and this is of fundamental importance for the sought applications because it enables a sensible enumeration of the topological dynamical states supported by a class of metamaterials. }$\Diamond$
\end{remark}

\begin{remark}{\rm For non-amenable groupoids, their reduced and full $C^\ast$-algebras are distinct. For the concrete applications discussed in section~\ref{Sec:Emp}, it is the reduced $C^\ast$-algebra that is relevant.}$\Diamond$
\end{remark}

\begin{example}\label{Ex:RA}{\rm Given the discussion in Example~\ref{Ex:RG}, it is immediate to see that, when the pattern is generated using a wallpaper group $G$ as in Examples~\ref{Ex:P2} and \ref{Ex:P4}, there is an isomorphism $C^{*}_{r}(\mathcal{G}_{\mathcal{L}_{0}},M_{N}(\mathbb{C}))\simeq M_N(\CM) \otimes C^\ast_r(G)$, where the latter is the reduced group $C^\ast$-algebra of $G$.
}$\Diamond$
\end{example}

\begin{remark}{\rm If the patterns from Examples~\ref{Ex:P2} and \ref{Ex:P4} are seen as closed subsets of $\RM^2$, as discussed in Remark~\ref {Re:GvsZ}, then the associated groupoid algebra will be  $M_{\alpha N}(\CM) \otimes C^\ast_r(\ZM^2)$, where $\alpha=2$ and 4, respectively. These $C^\ast$-algebras are strictly larger than the group $C^\ast$-algebras of the wallpaper groups that generated the patterns in the first place.}$\Diamond$
\end{remark}

\begin{example}{\rm The reduced groupoid $C^\ast$-algebra corresponding to the tranformation groupoid of a $G$-space coincides with the corresponding reduced crossed product $C^\ast$-algebra \cite[Example~1.51]{WilliamsBook2}. Therefore, for the setting discussed in Example~\ref{Ex:P2Disorder}, the canonical groupoid $C^\ast$-algebra reduces to the crossed product algebra associated to the dynamical system $\left({\rm p2},\left( B_\epsilon \times [-\epsilon,\epsilon] \right)^{\times \rm p2}\right)$.
}$\Diamond$
\end{example}

\begin{example}{\rm For the pattern from Example~\ref{Ex:1DQP}, the groupoid algebra reduces to the tensor product of $M_N(\CM)$ and the crossed product $C(\SM^1) \rtimes \ZM$, where the action of $\ZM$ is by rotation by $\theta$. 
}$\Diamond$
\end{example}

\subsection{The left-regular representations}
\label{Sec:LReps}

The restriction map $\rho$ introduced above extends by continuity to a positive map between $C^\ast$-algebras. By composing $\rho$ with the evaluation maps
\begin{equation}
j_{\Ll}: C\left(\Xi_{\Ll_0},M_N(\CM)\right) \to M_N(\CM), \ \ j_\Ll(\chi) = \chi(\Ll)
\end{equation}
for some $\Ll \in \Xi_{\Ll_0}$, one obtains a family of $M_N(\CM)$-valued positive maps. By further composing with the trace state on $M_N(\CM)$, one obtains a family of states
\begin{equation}
\rho_\Ll = {\rm Tr} \circ j_\Ll \circ \rho
\end{equation}
on $C^\ast_r\left(\Gg_{\Ll_0},M_N(\CM)\right) $, indexed by the transversal $\Xi_{\Ll_0}$. The GNS-representations corresponding to these states supply the left-regular representations of the $C^\ast$ algebra $C^\ast_r\left(\Gg_{\Ll_0},M_N(\CM)\right) $.

\subsection{Sorting dynamics with groupoid algebras}\label{Sec:SortingD}

We are now in the position to make the connection with the dynamical matrices~\eqref{Eq:FinalD} analyzed in section~\ref{Sec:Emp}. We recall that our stated goal is to determine the smallest $C^\ast$-algebra which generates all dynamical matrices associated to a given aperiodic architecture of identical resonators with adjustable internal structure.

We specialize the discussion to the particular case when the lcsc group $G$ is ${\rm Iso}(\EM^d)$ and $\Ll_0$ is one of its uniformly separated patterns, representing a specific configuration of identical resonators.

The Hilbert space corresponding to GNS-representation induced by $\rho_\Ll$ is 
\begin{equation}
\Hh_\Ll : =  \ell^2\left(\mathfrak s^{-1}(\Ll), \CM^N \right) = \ell^2\left(\Ll, \CM^N \right),
\end{equation}
the space of $\CM^N$-valued square summable sequences over $\Ll$, and the representation acts explicitly as
\begin{equation}\label{Eq:PiL}
[\pi_\Ll(f) \varphi](g') = \sum_{g \in \Ll}f\left(g \cdot (g',\Ll)\right)  \cdot \varphi (g), \quad g' \in \Ll, \ \varphi \in \ell^2\left(\Ll, \CM^N \right),
\end{equation}
which follows directly from Eq.~\eqref{Eq:Conv}. In particular,\footnote{Here, $\alpha \otimes |g'\rangle$ represents the map $\varphi(g) = \alpha \, \delta_{g,g'}$.}
\begin{equation}
\pi_\Ll(f)(\alpha \otimes |g'\rangle) = \sum_{g \in \Ll}f\left(g' \cdot (g,\Ll)\right)\cdot \alpha \otimes |g\rangle, \quad \alpha \in \CM^N.
\end{equation}
Below, we reproduce the expression~\eqref{Eq:FinalD} of a generic dynamical matrix
\begin{equation}\label{Eq:DL}
D_{\Ll} =\sum_{x,x' \in \Ll} w_{x' \cdot x,e}(x' \cdot \Ll) \otimes |x\rangle \langle x'|
\end{equation}
and, obviously,
\begin{equation}
D_{\Ll} (\alpha \otimes |x'\rangle)=\sum_{x \in \Ll} w_{x' \cdot x,e}(x' \cdot \Ll) \cdot \alpha \otimes |x\rangle.
\end{equation}
We now recall the discussion from subsection~\ref{Sec:Contr} and, by comparing Eqs.~\eqref{Eq:PiL} and \eqref{Eq:DL}, we see that a dynamical matrix with finite coupling range can be generated as the left-regular representation of an element $f$ from the core sub-algebra, if we take $f(g,\Ll) := w_{g,e}(\Ll)$ for any $(g,\Ll) \in \Gg_{\Ll_0}$. Since these sets are dense in both the space of acceptable dynamical matrices and the groupoid $C^\ast$-algebra, we can conclude that $C^\ast_r\left(\Gg_\Ll,M_N(\CM)\right)$ generate all dynamical matrices that can be controlled in a laboratory, for a fixed number of internal degrees of freedom.

Various dynamical effects, such as topological pumping, topological phase transitions, etc., require deformations of the internal structures of the resonators. During such deformations, the degrees of freedom carried by the seed resonator may be continuously turned on or off. Hence, we need to define a $C^\ast$-algebra that covers such situations and, furthermore, gives a precise sense to what it means to ``continuously'' turn on and off a degree of freedom. For this, we recall the isomorphism from Remark~\ref{Re:Iso7} and note the following canonical embedding 
\begin{equation}\label{eq: cpt-embed}
C^\ast_r(\Gg_{\Ll_0},M_N(\CM)) \simeq M_N(\CM) \otimes C^\ast_r(\Gg_{\Ll_0}) \hookrightarrow \KM \otimes C^\ast_r(\Gg_{\Ll_0}),
\end{equation}
where $\KM$ is the algebra of compact operators over $\ell^2(\Ll)$. Thus, $\KM \otimes C^\ast_r(\Gg_{\Ll_0})$ is a natural solution for the problem we just described. Furthermore, as argued in subsection~\ref{Sec:LAh}, this is the smallest $C^\ast$-algebra among the possible solutions, hence, it is the algebra we were actually looking for. At this point, we have justified the statements 1) through 5) in our introductory section~\ref{Sec:Introduction} and statement 6) is a direct consequence of those statements.

\section{Outlook}\label{Sec:Outlook}

This section indicates several directions in materials science and research that can be efficiently investigated with the formalism developed by our work.

Firstly, if the dynamical matrix $D_{\Ll_0}$ displays an isolated band $\Delta$ of resonant spectrum, then we can define the spectral projection $P_\Delta$ onto $\Delta$, usually called the band projection and observe the time evolution of a state $\psi$ excited from $P_\Delta \Hh_{\Ll_0}$. Since this time evolution cannot escape this subspace, $P_\Delta \mathcal H_{\Ll_0}$ defines an invariant for dynamics. As it is well known, these invariants are accounted for by the $K_0$-theory of $C^\ast_r(\Gg_{\Ll_0})$, and here is where the $C^\ast$-algebraic framework can make fine contributions to our understanding of the dynamics of these physical system \cite{Bellissard1986,Bellissard1995}. Our formalism opens such lines of research for contexts that were not explored yet.

Secondly, we already mentioned that the bulk-defect correspondence principle can be formulated very generally and precisely using groupoids, to a point where a full classification of materials defects can be attempted \cite{BourneAHP2020,ProdanJPA2021}. There is already plenty of evidence that possible point symmetries of defects can highly enhance the bulk-defect correspondences \cite{ChiuRMP2016}. For example, the point-defects of the patterns examined in Examples~\ref{Ex:P2} and \ref{Ex:P4}, can still display a symmetry relative to a point group. If this is the case, then our formalism leads to a defect $C^\ast$-algebra that is stably isomorphic to the group $C^\ast$-algebra of this point group. This means that the bulk-defect correspondence lands in the representation ring of this point group, thus possibly leading to interesting topological defect states. In contradistinction, in a formalism that only engages the group of pure translations,  the defect $C^\ast$-algebra is always isomorphic to $\KM$, hence no such refined predictions are possible. 

Thirdly, the mathematical framework can be applied to patterns of resonators generated with other lcsc groups, but one has to be careful on the physics side. For example, fractal patterns engage the scaling transformation of the Euclidean space. Thus, an interesting direction will be to upgrade from ${\rm Iso}(\EM^d)$ to the sub-group ${\rm Sim}(\EM^d)$ of the affine transformations generated by the isometries and dilation. In this case, however, the equivariance~\eqref{Eq:EquiV2} is no longer ensured by natural laws and, instead, it needs to be enforced by finely tuning the internal structure of the resonators, which, nevertheless, is within the reach of the present manufacturing capabilities. 

Lastly, the cyclic cohomology of the groupoid $C^\ast$-algebra and its pairing with the K-theory \cite{ConnesIHESPM1985} supplies the venue for us to discover the topological responses to external stimuli for entirely new classes of materials, {\it e.g.} by following \cite{BourneAHP2020,BourneJPA2018} as models of analysis.

\section{Appendix: The Euclidean spaces and their groups of isometries}\label{Appendix} 

At the conceptual level, the resonators live in the group of isometries of their ambient physical spaces. As such, we collect in this appendix some known facts about the algebraic and topological characteristics of these groups. The information will come handy when we discuss particular examples, though, it will play little role in the main formal developments.

\subsection{Generic facts} The $d$-dimensional Euclidean space $\EM^d$ consists of the set $\RM^d$ equipped with the metric
\begin{equation} 
{\rm d}(x,y) : = \left[\sum\nolimits_{j=1}^d(x_j -y_j)^2\right]^\frac{1}{2}.
\end{equation}
We will consistently represent a point of $\RM^d$ as a column vector with $d$ entries and use ``$\cdot$'' to indicate matrix multiplications. Furthermore, we let ${\rm Homeo}(\EM^d)$ denote the group of homeomorphims of $\EM^d$.

\begin{proposition}[\cite{ArensAJM1946}]  The set ${\rm Homeo}(\EM^d)$ equipped with the compact-open topology inherited from $\EM^d$ is a topological group, {\it i.e.} multiplication and inversion maps are both continuous in this topology.
\end{proposition}

\begin{remark}\label{Re:COTop}{\rm Certain fine aspects about the above statement are worth recalling. If a topological Hausdorff space is compact, then its group of homeomorphisms is automatically a topological group when equipped with the compact-open topology. This is not always the case for non-compact space. It is true in this particular case because the compact-open topology on ${\rm Homeo}(\EM^d)$ coincides with the one inherited from the one-point compactification $\alpha \EM^d$ of the Euclidean space \cite{DijkstraAMM2005}.
}$\Diamond$
\end{remark} 

\begin{definition}\label{Def:E2} The Euclidean group ${\rm Iso}(\EM^d)$ is defined as the topological sub-group of ${\rm Homeo}(\EM^d)$ containing the set of isometric homeomorphisms,
\begin{equation}
    \Iso(\EM^d )
    : = \left\{ f \in {\rm Homeo}\left( \EM^2\right) ~|~  {\rm d}( f(x), f(y) ) = {\rm d}(x, y)~ \forall \ x,y \in \EM^d 
    \right\}.
\end{equation}
By definition, the topology of ${\rm Iso}(\EM^d)$ is that inherited from ${\rm Homeo}(\EM^d)$.
\end{definition}

\subsection{Structure and topology of the Euclidean groups} The topological group $\RM^d$ acts by translations on $\EM^d$ and is a normal sub-group of ${\rm Iso}(\EM^d)$. The topological group $O(d)$ acts by rotations on $\EM^d$ and is also a sub-group of ${\rm Iso}(\EM^d)$, though not a normal one. It consists of the set of orthogonal matrices
\begin{equation}
O(d) = \{ \bm r \in M_{d}(\RM), \ \bm r \cdot \bm r^T = \bm r^T \cdot \bm r = I_{d} \}
\end{equation}
and the topology inherited from ${\rm Homeo}(\EM^d)$ coincides with the one induced by the metric
\begin{equation}
{\rm d}_O(\bm r, \bm r') = \| \bm r - \bm r' \|_{\rm HS},
\end{equation}
where $\| \cdot \|_{\rm HS}$ is the Hilbert-Schmidt norm. The Euclidean group can be presented as the semi-direct product
\begin{equation}
{\rm Iso}(\EM^d) = \RM^d \rtimes_\beta O(d), \quad \beta_{\bm r}(v) = \bm r \cdot v.
\end{equation}
Thus, an element of ${\rm Iso}(\EM^d)$ can be uniquely presented as a tuple $(v,\bm r)$, with $v \in \RM^d$ and $\bm r \in O(d)$, and  composition and inversion rules
\begin{equation}
(v_1,\bm r_1)(v_2,\bm r_2) = (v_1+\bm r_1 \cdot v_2,\bm r_1 \cdot \bm r_2), \quad (v,\bm r)^{-1} = (-\bm r^{-1}\cdot v, \bm r^{-1}).
\end{equation}
We will reserve the symbol $\bm 1$ for the neutral element $(0,I_d)$ of ${\rm Iso}(\EM^d)$. As a topological space, ${\rm Iso}(\EM^d) = \RM^d \times O(d)$, with the group appearing on the right reduced to a topological space.

Regarding the structure of $O(d)$, any orthogonal matrix has determinant ${\rm det}(\bm r) = \pm 1$ and, since the determinant is a continuous map, this implies that $O(d)$ has at least two connected components, the pre-images of $\pm 1$ through the determinant map, which turn out to be the only connected components of $O(d)$ \cite{GuptaMS2018}. Furthermore, since the determinant is a topological group morphism, the connected component of the neutral element $I_{d}$ is a normal sub-group of $O(N)$, called the special orthogonal group and denoted by $SO(N)$. This leads to the following split exact sequence
\begin{equation}\label{Eq:Seq1}
\{e\} \to SO(d) \to O(d) \rightarrow O(d)/SO(d) \simeq \ZM_2 \to \{e\}.
\end{equation}

\subsubsection{Dimension $d=2$} In this case, the generator of $\ZM_2$ can be embedded in $O(2)$ via the matrix $\bm p=\begin{pmatrix} 0 & 1 \\ 1 & 0 \end{pmatrix}$, which has a nontrivial action by conjugation. Thus, $O(2)$ is the internal semi-direct product $O(2) = SO(2) \rtimes \left\langle \bm p, I_{2 \times 2} \right\rangle$ and, as a topological space,
\begin{equation}
O(2) \simeq \SM^1 \times \{-1,+1\}.
\end{equation}
Hence, as a topological space,
\begin{equation}\label{Eq:E2Topo}
{\rm Iso}(\EM^2) \simeq \RM^2 \times \SM^1 \times \{-1,+1\}.
\end{equation}

\subsubsection{Dimension $d=3$} In this case, the exact sequence~\eqref{Eq:Seq1} leads to a direct product $O(3)=SO(3) \times \{I_{d\times d},-I_{d\times d}\}$ and to a topological space, $O(3) \simeq \RM \PM^3 \times \{-1,+1\}$. Hence, as a topological space,
\begin{equation}\label{Eq:E3Topo}
{\rm Iso}(\EM^3) \simeq \RM^3 \times \RM \PM^3 \times \{-1,+1\}.
\end{equation}

\end{document}